\documentclass[12pt]{amsart}

\usepackage{
amsfonts,
latexsym,
amssymb,
}

\usepackage[abbrev,lite,
shortalphabetic,
]{amsrefs}

\newcommand{\labbel}{\label}

\newtheorem{theorem}{Theorem}[section]
\newtheorem{lemma}[theorem]{Lemma}

\newtheorem{proposition}[theorem]{Proposition} 
 
\newtheorem{corollary}[theorem]{Corollary}

\newtheorem*{theorem*}{Theorem}
\newtheorem*{corollary*}{Corollary}
\newtheorem*{proposition*}{Proposition}

\theoremstyle{definition}

\theoremstyle{remark}
\newtheorem{remark}[theorem]{Remark}

\newtheorem*{remark*}{Remark}

\DeclareMathOperator{\cf}{cf}

\DeclareMathOperator{\CAP}{CAP}

\begin{document}
 
\title
{On a theorem by Juh{\'a}sz and Szentmikl{\'o}ssy}

\author{Paolo Lipparini} 
\address{Nuevo Departamento de Matematica\\ Viale della Ricerca Scientifica\\II Universit\`a di Roma (Tor Vergata)\\I-00133 ROME ITALY}
\urladdr{http://www.mat.uniroma2.it/\textasciitilde lipparin}



\begin{abstract} 
We extend a theorem by 
Juh{\'a}sz and Szentmikl{\'o}ssy  to 
notions related to pseudocompactness.
We also allow the case
when one of the cardinals
under consideration  is singular.

We give an application to the study of decomposable ultrafilters:
if $\kappa$ is singular,
$D$ is a uniform ultrafilter over $ \kappa ^+ $, and
$D'$ is a uniform ultrafilter over $ \cf \kappa $,
then $D' \times D$ is $\kappa$-decomposable. 
\end{abstract}

\maketitle

\section{Introduction} \labbel{intro} 

Juh{\'a}sz and Szentmikl{\'o}ssy  \cite{JS} 
introduced the principle $\Phi( \mu, \kappa, \lambda )$, and used it 
to get some ``interpolation'' theorems for $\kappa$-compactness of
topological spaces; in more detail, 
they proved that if $\Phi( \mu, \kappa, \lambda )$ holds and a space is both 
$\mu$-compact and $\lambda$-compact, then it is also $\kappa$-compact.
They also obtained many further interesting consequences.

In what follows $\mu $, $\kappa$ and $\lambda$ shall always
assumed to be infinite cardinals; as usual,
 $[ S]^ {< \nu} = \{ Z \subseteq S \mid |Z| < \nu\}  $; similarly,
 $[ S] ^{ \nu} = \{ Z \subseteq S \mid |Z| = \nu\}  $. The principle $\Phi( \mu, \kappa, \lambda )$ introduced by
Juh{\'a}sz and Szentmikl{\'o}ssy  is the assertion that
there exists a family
$\{  S_ \xi \mid \xi < \lambda \}  \subseteq [  \kappa ] ^{ \mu} $
such that 
$|\{ \xi < \lambda \mid | A \cap S_ \xi| = \mu \}|< \lambda $,
for every $A \in [ \kappa ] ^{ < \kappa }$. 
Notice that Juh{\'a}sz and Szentmikl{\'o}ssy assumed that
$\lambda$ is regular in the above definition;
we shall not need this assumption.

In Section \ref{pseudo} we show, by a small modification of 
Juh{\'a}sz and Szentmikl{\'o}ssy  argument,  that, when $\lambda$ is not assumed
to be regular, the above mentioned result still holds under the additional assumption of 
$\cf\lambda$-compactness. 
Perhaps more significantly, we present a version 
which deals with properties connected with pseudocompactness. 

In Section \ref{decsec} we show that if $\Phi( \mu, \kappa, \lambda )$ holds,
$D'$ is a uniform ultrafilter over $\mu $, and 
$D$ is  a uniform ultrafilter over $\lambda$,
then $D \times D'$ is $\kappa$-decomposable, together with
some related results. We expect to be able to present
even deeper connections among Juh{\'a}sz and Szentmikl{\'o}ssy principle,
and regularity and decomposability of ultrafilters in the nearest future.

\section{$\kappa$-compactness relative to some family $\mathcal F$} \labbel{pseudo} 

First, some definitions are needed. 
We shall consider a topological space $X$ together 
with a family $\mathcal F$ of nonempty subsets of $X$.
This is done in order to treat 
simultaneously the following two main cases: 
$\mathcal F =\mathcal S$, the set of all singletons
of $X$, in which case we get notions and theorems related to
compactness, $\kappa$-compactness, \dots, 
and the case when 
$\mathcal F =  \mathcal O$, the set of all nonempty open
subsets of $X$, in which case we get notions related to pseudocompactness.
At first reading, the reader is advised  to consider
only the above examples (or the example he or she is more interested in).

If   
$X$ is a topological space,
and $\mathcal Y $ is an infinite subset of $ \mathcal P(X) $ 
 we say that  $x \in X$ is a
\emph{complete accumulation point}  
of
$\mathcal Y$ 
if  
$ | \{ Y  \in \mathcal Y \mid Y \cap U \not= \emptyset  \} |
=| \mathcal Y|$,
for every  neighborhood $U$ of $x$ in $X$. 
When all the members of $\mathcal Y$ are singletons, 
we get the usual notion of a \emph{complete accumulation point}
of a subset $Y$ of $X$ (where, in this case
 $Y =  \bigcup  \mathcal Y $).  
If $\mathcal F \subseteq \mathcal P(X)$, we say that 
$X$ is \emph{$\mathcal F$-$ \lambda $-compact}
if   every subfamily $ \mathcal Y \subseteq \mathcal F$ of cardinality $\lambda$ 
has a complete accumulation point.  
When $\mathcal F = \mathcal S$, this is called \emph{$ \lambda $-compactness},
and means that every subset of $X$ of cardinality
$\lambda$ has a complete accumulation point.
The other interesting case is when
$\mathcal F =   \mathcal O$.
Just to mention an example, for Tychonoff spaces,
$\mathcal O$-$\omega  $-compactness in the above sense
is equivalent to pseudocompactness.
We refer to 
\cite{tproc2}, in particular, Section 3,
for more details about 
$\mathcal F$-$ \lambda $-compactness
(called there $\mathcal F$-$\CAP _ \lambda $),
equivalent formulations, related notions,
references to the literature
as well as references to alternative terminology used in the literature.
Throughout the present section, we shall assume that $X$ is some fixed topological space,
and $\mathcal F$ is a fixed family such that $   \emptyset  \not= \mathcal F \subseteq \mathcal P(X) \setminus \{ \emptyset  \} $.
 
\begin{theorem} \labbel{js2}
Suppose that $\Phi( \mu, \kappa, \lambda )$ holds, 
that 
$X$ is $\mathcal F$-$ \mu $-compact and
$\lambda$-compact and, in case $\lambda$ is
singular, suppose further that $X$ is $\cf\lambda$-compact. 
Then $X$ is $\mathcal F$-$ \kappa $-compact.
\end{theorem} 
 
\begin{proof}
The proof follows essentially the lines of the proof 
of \cite[Theorem 2]{JS} with a small variation in the case $\lambda$ singular. 
Let
$\mathcal Y
\subseteq \mathcal F$ be a family 
of cardinality $ \kappa $ and
enumerate it as
 $ \{ Y _ \alpha  \mid  \alpha \in \kappa \} $
with all the $Y_ \alpha $'s
 distinct. Let  
$\{  S_ \xi \mid \xi < \lambda \} \subseteq [ \kappa ]^ \mu $
be given by $\Phi( \mu, \kappa, \lambda )$.
By $\mathcal F$-$ \mu $-compactness, 
for every $\xi < \lambda$, the family
 $ \{ Y _ \alpha  \mid  \alpha \in S_ \xi \} $ 
(of cardinality $\mu $)
has a complete accumulation point $p _ \xi$,
that is, 
$|\{ \alpha \in S_ \xi \mid Y_ \alpha \cap U \not= \emptyset  \}|= \mu$,
for every neighborhood $U$ of $p_ \xi$.

If $\lambda$ is regular, there are two cases:
(1) there is $p \in X$ such that 
$|\{ \xi < \lambda \mid p _ \xi =p\}| = \lambda $;
(2) $|\{ p _ \xi \mid \xi < \lambda \}| = \lambda $.
In this latter case, choose some complete accumulation point
$p$ of $\{ p _ \xi \mid \xi < \lambda \}$; the existence 
of such a $p$ is guaranteed
by $\lambda$-compactness.
If $\lambda$ is singular, a third case can occur:
(3) for every $\lambda' < \lambda $,
there is $p ( \lambda ') \in X$ such that 
$ \lambda ' \leq |\{ \xi < \lambda \mid p _ \xi =
p( \lambda ')\}| < \lambda $. In this case, fix
an increasing sequence 
$ (\lambda_ \gamma ) _{ \gamma \in \cf \lambda } $
cofinal in $\lambda$, and, for every 
$\gamma \in \cf \lambda$,
choose some 
$ p _ \gamma = p( \lambda _ \gamma )$ as above. 
 Then necessarily
$|\{ p _ \gamma  \mid \gamma \in \cf \lambda  \}| = \cf \lambda $,
and, by $\cf\lambda$-compactness,
the set $\{ p _ \gamma  \mid \gamma \in \cf \lambda  \}$ has a complete accumulation point
$p$.

In each of the above cases, for every neighborhood 
$U$ of $p$, we have that 
 $ | \{ \xi < \lambda |
 p_ \xi \in U
 \} | = 
\lambda $,
hence
 $ | \{ \xi < \lambda |
| A^U \cap S_ \xi  | = \mu
 \} | = 
\lambda $,
where we have put
$A^U =\{ \alpha \in \kappa  \mid Y_ \alpha \cap U \not= \emptyset  \}$.
Here we have used the fact that
if $p_ \xi \in U$, $U$ a neighborhood of $p$, then $U$ is also a neighborhood of $p_\xi$. 
We show that $p$ is a complete  accumulation point 
of $ \{ Y _ \alpha  \mid  \alpha \in \kappa \}$.
Suppose not.  Then, for some neighborhood $U$
of $p$,  $|A^U|
<\kappa$.  
Then 
$\Phi( \mu, \kappa, \lambda )$
implies that
$|\{ \xi < \lambda \mid | A^U \cap S_ \xi| = \mu \}|< \lambda $,
a contradiction.
 \end{proof}  

Notice that if in the above theorem $\lambda$ is a regular cardinal, 
and $\mathcal F$ is taken to be equal to $\mathcal S$,
then we get exactly the statement of \cite[Theorem 2]{JS}.
The arguments in the proof of Theorem \ref{js2} actually give the following stronger 
``local'' version. 

\begin{corollary} \labbel{local}
Suppose that $\Phi( \mu, \kappa, \lambda )$ holds, 
that 
$X$ is 
$\lambda$-compact and, in case $\lambda$ is
singular, suppose further that $X$ is $\cf\lambda$-compact. 

If $\mathcal Y \subseteq \mathcal P(X) $,
$|\mathcal Y|= \kappa $, and every subset of $\mathcal Y$
of cardinality $\mu $ has a  
 complete accumulation point,
then $\mathcal Y$ has a complete accumulation point.
In particular, if   $ Y \subseteq   X $,
$|Y|= \kappa $, and every subset of $ Y$
of cardinality $\mu $ has a  
 complete accumulation point,
then $  Y$ has a complete accumulation point.
 \end{corollary}  

Using the theorems and arguments in \cite{JS}, and by  Theorem \ref{js2},
we get the following corollary (notice that if $ \emptyset \not\in \mathcal F$,
then
$\kappa$-compactness trivially implies $\mathcal F$-$\kappa$-compactness).

\begin{corollary} \labbel{cor} 
If $X$ is  linearly Lindel\"of and 
$\mathcal F$-$ \aleph_ \omega $-compact, 
then $X$  is $\mathcal F$-$ \kappa $-compact,
for every uncountable cardinal $ \kappa $. 
\end{corollary}

\begin{remark} \labbel{rmk} 
Let us remark that in the present section we have used very little topology.
In all the above arguments (and in a large part of the paper by
Juh{\'a}sz and Szentmikl{\'o}ssy)
the only needed assumption is that 
to every point $x$ of $X$ there is associated a family
of ``neighborhoods'' with the only properties that
(1) each ``neighborhood'' of $x$ contains $\{ x \}$, and (2)  
if $U$ is a ``neighborhood'' of $x$, then
every $y \in U$ has some ``neighborhood''
contained in $U$.
In particular, we have never used the topological
property that the intersection of two neighborhoods
of $x$ is still a neighborhood of $x$.
\end{remark}

\section{Decomposability of products of ultrafilters} \labbel{decsec}

In this section  we  use $\Phi( \mu, \kappa, \lambda )$ to prove theorems 
about decomposability of ultrafilters. In fact, the proofs
work without changes for arbitrary families of subsets of some set, not even necessarily
being filters; so we state the results in such a generality.
Of course, the reader may always assume 
that $D$, $D'$ below are ultrafilters, and it might happen
this is the only interesting case.

We say that $D \subseteq \mathcal P ( I ) $ 
 is \emph{uniform over $I$} if every member of $D$ has cardinality
$|I|$.  Here $ \mathcal P (I)$ is the set of all subsets of $I$.
 We say that
$D$ is \emph{$\mu $-decomposable} if
 there exists a function $f:I \to \mu $ such that 
$f ^{-1} (Z) \not\in D$, whenever $Z\in [ \mu ] ^{< \mu} $.
Such an $f$ is called a \emph{$\mu $-decomposition} for $D$.  
Clearly, $f$ is  a $\mu $-decomposition  for $D$ if and only 
if $f(D)$ is uniform over $\mu $, 
where we put $f(D)=\{ Z \subseteq  \mu \mid f ^{-1} (Z) \in D \}$.  
In the above definitions, $\mu $ can be equivalently
replaced by any set of cardinality $\mu $.
Throughout the paper,  
$D $, $D'$ are assumed to be subsets of
 $   \mathcal P (I)$,   $   \mathcal P (I')$, respectively,

The next theorem exploits the connection
between $\Phi( \mu, \kappa, \lambda )$  and decomposability of families of sets.
As in the preceding section, we are not necessarily assuming that $\lambda$ is regular, though we do not know how much
this is an actual  gain in generality.
We define
the \emph{product} $D\times D' $
to be the subset of 
$ \mathcal P(I\times I')$
defined by:
$X\in D\times D'$ if and only if $ \{i\in I| \{i'\in I'|(i,i')\in X \}\in D'\}\in D $.
Of course, for ultrafilters, this coincides with a classical definition. 
See \cite{kat68} and further references there. 

\begin{theorem} \labbel{main} 
If $\Phi( \mu, \kappa, \lambda )$ holds,
$D'$ is  $\mu $-decomposable, and
$D$ is  $\lambda$-decomposable,
then $D \times D'$ is $\kappa$-decomposable. 
\end{theorem}

 \begin{proof}
We first prove the theorem in the particular case in which
$D'$ is  uniform  over $\mu $ and
$D$ is  uniform  over $\lambda$.
Let 
$\{  S_ \xi \mid \xi < \lambda \} $
be given by $\Phi( \mu, \kappa, \lambda )$ and,
for every $\xi < \lambda $,
choose a bijection $f _ \xi: \mu \to S_ \xi$. 
Since $D'$ is uniform over $\mu $,
$f_ \xi ^{-1}(C) \not\in D' $,
for every $\xi < \lambda $ and $C \in [S_ \xi] ^{<\mu} $.
Define $f: \lambda  \times \mu \to \kappa  $
by $f( \xi, \eta) = f_ \xi (\eta)$.
We claim that $f$ witnesses that    
$D \times D'$ is $\kappa$-decomposable.
Indeed, let 
$A \in [ \kappa ] ^{< \kappa } $.
By $\Phi( \mu, \kappa, \lambda )$,
$|\{ \xi < \lambda \mid | A \cap S_ \xi| = \mu \}|< \lambda $,
thus, from the above remark,  
$|\{ \xi < \lambda \mid  f_ \xi ^{-1}( A) \in D' \}|< \lambda $,
hence, since $D$ is uniform,
$\{ \xi < \lambda \mid  f_ \xi ^{-1}( A) \in D' \} \not \in D$,
that is, 
$\{ \xi < \lambda \mid  \{  \eta \mid  f(\xi, \eta) \in A\}  \in D' \} \not \in D$,
and this means exactly
$f ^{-1} (A) \not\in  D \times D'$.

Now consider the general case in which $D'$ is  $\mu $-decomposable, and
$D$ is  $\lambda$-decomposable.
Then there are $g:I \to \lambda $ and 
 $g':I' \to \mu $
such that 
$g(D)$ is uniform over $\lambda$ 
and $g'(D')$ is uniform over $\mu $.
Applying the above particular case to 
$g(D)$ and $g'(D')$,
we get some 
$f: \lambda  \times \mu \to \kappa  $
 witnessing that    
$g(D) \times g'(D')$ is $\kappa$-decomposable.
Then clearly the function
$h: I \times I' \to \kappa  $
defined by
$h(i,i')=f(g(i),g'(i'))$ 
witnesses that    
$D\times D'$ is $\kappa$-decomposable.
 \end{proof}  

The same proof as above provides a slightly more general result.
If
$D  \subseteq \mathcal P (I)$, and, for every $i \in I$,
$D_i  \subseteq \mathcal P (I_i)$,
the \emph{$D$-sum} $\sum_D D_i$ of the $D_i$'s
modulo $D$ is the   subset  of
$ \mathcal P( \{(i,j)|i\in I , j\in I_i  \})  $
defined by
$X\in \sum_D D_i$ if and only if $ \{i\in I| \{j \in I_i| (i,j)\in X    \} \in D_i       \} \in D $.

\begin{theorem} \labbel{sums} 
If $\Phi( \mu, \kappa, \lambda )$ holds,
$D  \subseteq \mathcal P (I)$ is   $\lambda$-decomposable 
and, 
 for every $i \in I$,
$D_i  $ is 
$\mu $-decomposable
(or just $\{ i \in I  \mid D_i  \text{ is 
$\mu $-decomposable}\} \in D$), then $\sum_D D_i$ is $\kappa$-decomposable. 
\end{theorem}  

\begin{corollary} \labbel{l+} 
If $\kappa$ is singular,
$D'$ is  $ \cf \kappa $-decomposable, and
$D$ is  $ \kappa ^+$-decomposable,
then $D \times D'$ is $\kappa$-decomposable. 
\end{corollary}

 \begin{proof} 
Immediate from Theorem \ref{main} and 
and Juh{\'a}sz and Szentmikl{\'o}ssy Theorem 4
in \cite{JS}, asserting that
    $\Phi( \cf \kappa , \kappa, \kappa ^+ )$ holds. 
\end{proof} 

Notice that  Juh{\'a}sz and Szentmikl{\'o}ssy also proved that,
for example, $\Phi( \mu , \kappa, \kappa ^+ )$ holds
whenever $\cf \kappa = \cf \mu < \mu < \kappa $.
See \cite[Theorem 5]{JS}. However 
this adds nothing to the  theorems of the present section, since 
$\mu $-decomposability implies $\cf \mu $-decomposability.

\section{Further Remarks} \labbel{furth} 

We add a very simple observation, which nevertheless might be of some interest.
It elaborates on a classical argument, which dates back at least to Kat{\v{e}}tov \cite{kat68}. See the last lines on p. 173 therein.

If $D \subseteq \mathcal P(I)$,
$X$ is a topological space,
and $(F_i) _{i \in I} $  is a sequence of subsets of $X$,
then a point $p \in X$ is said to be a 
\emph{$D$-limit point}
(or a \emph{$D$-accumulation point}) of
$(F_i) _{i \in I} $
if 
$\{ i \in I \mid U \cap F_i \not= \emptyset \} \in D$,
for every neighborhood $U$ of $p$. In case each 
$F_i $  is a singleton $ \{ x_i\} $, we shall 
simply say that $p$ is  a $D$-limit point of $(x_i) _{i \in I} $. 
In this situation, it is sometimes said
 that $(x_i) _{i \in I} $
\emph{$D$-converges} to $p$.

If $\mathcal F \subseteq \mathcal P(X)$,  
 we say that $X$ is \emph{$\mathcal F$-$D$-compact}
(or that $X$ satisfies the \emph{$\mathcal F$-$D$-accumulation property})
if,  for every sequence
$(F_i) _{i \in I} $ of (not necessarily distinct) members of $\mathcal F$,
there is $p \in X$ which is a $D$-limit point of $(F_i) _{i \in I} $.
When $\mathcal F = \mathcal S$ this is simply called
\emph{$D$-compactness};    
when $\mathcal F = \mathcal O$ this is   called
\emph{$D$-pseudocompactness} (sometimes, \emph{$D$-feeble 
compactness}). 
Again, see \cite{tproc2} for further comments and references.
The above definitions are usually given only in the particular case when $D$
is an ultrafilter, but here, as in Section \ref{decsec}, we shall need no particular assumption 
on $D$.

\begin{lemma} \labbel{lem1} 
Suppose that $D \subseteq \mathcal P(I) $, that
 $D_i \subseteq \mathcal P(I_i) $, for each $i \in I$, 
and that $(F _{i, j} ) _{i \in I, j \in I_i} $
is a sequence of  subsets of $X$. If, 
for each $i \in I$, the subsequence
 $(F _{i, j} ) _{ j \in I_i} $ has some $D_i$-limit point $p_i$,
 and if the sequence $(p_i) _{i \in I} $  has a $D$-limit point
$p$, then $p$
is also a $\sum_D D_i$-limit point of 
 $(F _{i, j} ) _{ j \in I_i} $.
\end{lemma}

 \begin{proof}
Applying the definitions, we have that
$\{  j \in I_i \mid U \cap F _{i, j} \not= \emptyset \} \in D_i$,
for every $i \in I$ and
 every neighborhood $U_i$ of $p_i$; moreover, 
$\{ i \in I \mid p_i \in U \} \in D$, 
for every neighborhood $U$ 
of $p$.
The above 
statements, together
with the fact that if $p_i \in U$, then $U$ is also a neighborhood of $p_i$,
 give $ \{(i,j)|i\in I , j\in I_i,  U \cap F _{i, j} \not= \emptyset \} 
\in \sum_D D_i $.
 \end{proof}

\begin{proposition} \labbel{prop}
If 
 $X$ is both
$D$-compact and 
$\mathcal F$-$D_i$-compact,
for every $i \in I$,
then $X$ is 
$\mathcal F$-$\sum_D D_i$-compact.

In particular, if 
 $X$ is  
$D$-compact and 
$\mathcal F$-$D'$-compact,
then $X$ is 
$\mathcal F$-$D \times D'$-compact.
 \end{proposition} 

 \begin{proof} 
The second statement is a particular case of the first one, so let us prove the first statement.
Suppose that $(F _{i, j} ) _{i \in I, j \in I_i} $
is a sequence of elements of $\mathcal F$. 
By  $\mathcal F$-$D_i$-compactness, for every $i \in I$, the subsequence
 $(F _{i, j} ) _{ j \in I_i} $ has some
$D_i$-limit point 
$p_i $.
By $D$-compactness,  the sequence
$(p_i) _{i \in I} $, has a $D$-limit point, and we are done by the previous lemma.
\end{proof} 

Notice that Remark \ref{rmk} applies to  Lemma \ref{lem1} and Proposition \ref{prop}, too. 

\begin{remark} \labbel{apprec}   
Proposition \ref{prop} can be used to appreciate the connections
between Theorem \ref{js2}  and Theorem \ref{main}. 
Indeed, suppose that $ \lambda $ is a regular cardinal. Then
it is easy to see that $ \lambda $-compactness is equivalent to 
  $D$-compactness, for $D= [ \lambda ] ^{ \lambda } $.
 Similarly, if also $\mu $ is regular, then
$ \mu $-compactness is equivalent to 
  $D'$-compactness, for $D'= [ \mu] ^{ \mu } $.
From Theorem \ref{main} we get that if
$\Phi( \mu, \kappa, \lambda )$ holds,
then $D \times D'$ is $\kappa$-decomposable, that is,
 $D''=f(D \times D')$ is uniform over $\kappa$, for some appropriate
function $f$.
If, under the above assumptions,  $X$ is both
$\lambda$-compact and $\mu $-compact, then,
by Proposition \ref{prop}, $X$ is also
$D \times D'$-compact, and it is trivial to see that
then $X$ is also $D''$-compact, and this implies that
$X$  is $\kappa$-compact.

More generally, for $\mu $ regular, 
$\mathcal F$-$ \mu$-compactness is equivalent to 
$\mathcal F$-$D'$-compactness, for the same $D'$ as above, and
all the above arguments work for arbitrary $\mathcal F$. 
Hence, in the case when $\mu $ is regular, 
Theorem \ref{js2} is in fact a corollary of Theorem \ref{main}
and of Proposition \ref{prop}.
We do not need the assumption that
 $ \lambda  $ is regular, since,
in the case when $\lambda$ is singular,   $D$-compactness (where, 
as above, $D= [ \lambda ] ^{ \lambda  } $) is equivalent to the conjunction 
of both $ \lambda  $-compactness and of $\cf \lambda  $-compactness (see
\cite[Proposition 3.3]{tproc2}). 

On the other hand, 
some technical differences arise when $\mu $ is singular.
 Indeed, notice that Juh{\'a}sz and Szentmikl{\'o}ssy \cite{JS}
give significant applications of 
$\Phi( \mu, \kappa, \lambda )$
also in the case when $\mu $ is singular,
$\mu $-compactness holds, but  $\cf\mu $-compactness does not.
Compare also Corollary \ref{cor} in the present paper. 
 \end{remark}

\def\cprime{$'$} \def\cprime{$'$}

\end{document}